\documentclass[12pt, reqno]{amsart}

\usepackage{amssymb}
\usepackage[abbrev]{amsrefs}
\usepackage{mathrsfs}
\usepackage{bm}
\usepackage{enumitem}

\newtheorem{thm}{}[section]
\newtheorem{theorem}[thm]{Theorem}
\newtheorem{corollary}[thm]{Corollary}
\newtheorem{lemma}[thm]{Lemma}
\newtheorem{proposition}[thm]{Proposition}

\theoremstyle{definition}

\theoremstyle{remark}

\numberwithin{equation}{section}
\allowdisplaybreaks

\newcommand{\bphi}{\ensuremath{\bm{\phi}}}
\newcommand{\bpsi}{\ensuremath{\bm{\psi}}}
\newcommand{\df}{\ensuremath{\bm{\varphi}}}
\newcommand{\udf}{\ensuremath{\bm{\varphi_u}}}
\newcommand{\ldf}{\ensuremath{\bm{\varphi_l}}}
\newcommand{\pw}{\ensuremath{\bm{\Upsilon}}}
\newcommand{\ttt}{\ensuremath{\bm{t}}}
\newcommand{\sss}{\ensuremath{\bm{s}}}
\newcommand{\ww}{\ensuremath{\bm{w}}}
\newcommand{\uu}{\ensuremath{\bm{u}}}
\newcommand{\yy}{\ensuremath{\bm{y}}}
\newcommand{\xx}{\ensuremath{\bm{x}}}
\newcommand{\Ind}{\ensuremath{\mathbf{1}}}
\newcommand{\EE}{\ensuremath{\mathbb{E}}}
\newcommand{\BB}{\ensuremath{\mathcal{B}}}
\newcommand{\GG}{\ensuremath{\mathcal{G}}}
\newcommand{\YB}{\ensuremath{\mathcal{Y}}}
\newcommand{\XB}{\ensuremath{\mathcal{X}}}
\newcommand{\Ts}{\ensuremath{\mathcal{T}}}
\newcommand{\UU}{\ensuremath{\mathcal{R}}}
\newcommand{\SL}{\ensuremath{\mathscr{L}}}
\newcommand{\TT}{\ensuremath{\mathbb{T}}}
\newcommand{\NN}{\ensuremath{\mathbb{N}}}
\newcommand{\FF}{\ensuremath{\mathbb{F}}}

\newcommand{\XX}{\ensuremath{\mathbb{X}}}
\newcommand{\ZZ}{\ensuremath{\mathbb{Z}}}
\newcommand{\YY}{\ensuremath{\mathbb{Y}}}
\newcommand{\WW}{\ensuremath{\mathbb{W}}}
\newcommand{\VV}{\ensuremath{\mathbb{V}}}

\DeclareMathOperator{\spn}{span}
\DeclareMathOperator{\sgn}{sign}
\DeclareMathOperator*{\Ave}{Ave}

\let\abs=\envert

\let\norm=\enVert

\AtBeginDocument{
\def\MR#1{}
}

\begin{document}

\title[Fundamental functions of almost greedy bases of  $L_p$]{Fundamental functions of almost greedy bases of  $\bm{L_p}$ for $\bm{1<p< \infty}$}

\author[J.~L. Ansorena]{Jos\'e L. Ansorena}
\address{Department of Mathematics and Computer Sciences\\
Universidad de La Rioja\\
Logro\~no\\
26004 Spain}
\email{joseluis.ansorena@unirioja.es}



\begin{abstract}
We prove that the fundamental function of any almost greedy basis of $L_p$, $1<p<\infty$, grows as either $(m^{1/p})_{m=1}^\infty$ or $(m^{1/2})_{m=1}^\infty$.
\end{abstract}

\thanks{The author acknowledges the support of the Spanish Ministry for Science, Innovation, and Universities under Grant PGC2018-095366-B-I00 for \emph{An\'alisis Vectorial, Multilineal y Aproximaci\'on}.}


\maketitle

\section{Introduction}
\noindent The study from a functional analytic point of view of greedy-like bases in Banach spaces sprang from the celebrated characterization of greedy bases as those bases that are simultaneously unconditional and democratic \cite{KoTe1999}. Since Konyagin and Temlyakov's foundational result, several types of bases of interest in approximation theory that can be characterized by combining derived forms of unconditionality and democracy have appeared in the literature.

Apart from greedy bases, the most relevant types of bases arisen from the study of the thresholding algorithm are almost greedy ones and quasi-greedy ones. Almost greediness is a weakened form of greediness and, in turn, quasi-greediness is a weakened form of unconditionality. In some sense the role played by quasi-greediness within the study of almost greedy bases runs parallel to the role played by unconditionality within the study of greedy ones. In fact, it is known that a basis is almost greedy if and only if it is quasi-greedy and democratic (see \cite{DKKT2003}*{Theorem 3.3} and \cite{AABW2021}*{Theorem 6.3}).

An intrinsic characteristic of democratic bases is its fundamental function. So, from a functional analytic point of view, it is very natural to ask in which way the geometry of the space affects the fundamental functions of almost greedy bases and greedy ones. As far as greedy bases are concerned, this topic connects with that of uniqueness of greedy bases in Banach spaces. Suppose that a Banach space $\XX$ has a greedy basis and that any greedy basis of $\XX$ is equivalent to a permutation of $\BB$. Then, obviously, all the greedy bases of $\XX$ have, up to equivalence, the same fundamental function. Within the class of Banach spaces with a unique (up to equivalence and permutation) greedy basis, we must differentiate two disjoint subclasses. On the hand, if $\XX$ has a unique (semi-normalized) unconditional basis $\BB$ which is democratic, then $\BB$ also is the unique greedy basis of $\XX$. One the other hand, there are Banach spaces with a unique greedy basis that also possess non-democratic semi-normalized unconditional bases. To the former subclass belong the spaces $\ell_1$, $\ell_2$, $c_0$, the Tsirelson space $\Ts$, and the $2$-convexified Tsirelson space $\Ts^{(2)}$ (see \cites{KotheToeplitz1934,LinPel1968, LinZip1969,BCLT1985,CasKal1998}). To the latter subclass belong certain Orlicz sequence spaces near either to $\ell_1$ or to $\ell_2$ and the separable part of weak $\ell_p$ for $1<p<\infty$ (see \cite{AADK2016}). If we broaden the scope of the study to nonlocally quasi-Banach spaces, the former subclass enlarges considerably. In fact, given $p\in(0,1)$, $\ell_p$, Hardy space $H_p(\TT)$, Lorentz sequence spaces $\ell_{p,q}$ for $0<q\le\infty$, and Orlicz sequence spaces $\ell_F$ for a concave Orlicz function $F$ belong to it (see \cites{Kalton1977b,KLW1990,Woj1997}).

There also exist Banach spaces without a unique greedy basis in which all greedy bases have the same fundamental function. For instance, the sequence space $\ell_p$ for $p\in(1,2)\cup(2,\infty)$ has a continuum of mutually permutatively non-equivalent greedy bases (see \cites{DHK2006,Smela2007}), and the fundamental function of all of them grows as
\[
\pw_p:=(m^{1/p})_{m=1}^\infty.
\]
More generally, it is known that given $0<p\le \infty$, the fundamental function of any super-democratic (that is, democratic and unconditional for constant coefficients) basis of $\ell_p$ (we replace $\ell_p$ with $c_0$ if $p=\infty$) is equivalent to $\pw_p$ (see \cite{AADK2019b}*{Proposition 4.21}). Therefore, the fundamental function of any almost greedy basis of $\ell_p$ is equivalent to $\pw_p$. We draw reader's attention that any super-democratic basis whose fundamental function is equivalent to $\pw_\infty$ is equivalent to the unit vector system of $c_0$. This observation gives that $c_0$ has a unique almost greedy basis. In contrast, $\ell_2$ and $\ell_p$ for $0<p\le 1$, despite having a unique greedy basis, have a continuum of mutually permutatively non-equivalent almost greedy bases (see \cite{DHK2006}*{Theorem 3.2} and \cite{AAW2021b}*{Corollary 6.2}). As far as quasi-greedy bases are concerned, the main structural difference between the spaces $\ell_p$ for $p\in (0,1]\cup\{2,\infty\}$ and the spaces $\ell_p$ for $p\in(1,2)\cup(2,\infty)$ is that, unlike the former, the latter spaces have non-democratic quasi-greedy bases (see \cites{Woj2000,DKK2003,DSBT2012,AAW2021}).

In $\ell_p$-spaces, the fundamental functions of greedy bases behave like those of almost greedy ones. However, a priori, the geometry of the space provides less information on the fundamental function of almost greedy bases than that it does on the fundamental function of greedy bases. An important example of this situation is the Lebesgue $L_p=L_p([0,1])$ for $1<p<\infty$. Since any unconditional basis of $L_p$ possesses a subbasis equivalent to the unit vector system of $\ell_p$ (see \cite{KadPel1962}), the fundamental function of any greedy basis of $L_p$ grows as $\pw_p$. In this sense, the behavior of greedy bases in $L_p$, $1<p<\infty$, runs parallel to that of greedy bases in $\ell_p$. This parallelism breaks down when dealing with almost greedy bases. To see this, we bring up Nielsen' paper \cite{Nielsen2007}, in which the author constructs a uniformly bounded orthogonal system of $L_2$ that is an almost greedy basis for $L_p$ for each $1<p<\infty$. Although not explicitly stated by the author, the proof of \cite{Nielsen2007}*{Theorem 1.4} gives that the fundamental function of the achieved basis of $L_p$ grows as $\pw_2$. This fact is not casual: by \cite{AACV2019}*{Proposition 2.5}, any quasi-greedy basis $\Psi$ of $L_p$ with $\sup_n \Vert \bpsi_n \Vert_\infty<\infty$ is democratic with fundamental function of the same order as $\pw_2$. Hence, $L_p$, $1<p<\infty$ and $p\not=2$, has almost greedy bases whose fundamental functions grow differently. As for $L_2$, we point out that any quasi-greedy basis of a Hilbert space is democratic with fundamental function equivalent to $\pw_2$ (see \cite{Woj2000}*{Proof of Theorem 3}). As the case $p=1$ is concerned, we bring up \cite{DSBT2012}*{Theorem 4.2}, which implies that all quasi-greedy bases of $L_1$ and $\ell_1$ are democratic with fundamental function of the same order as $\pw_1$. Moreover, by \cite{DHK2006}*{Theorem 3.2}, there is a continuum of mutually non-permutatively equivalent quasi-greedy bases of $L_1$. Thus, in some sense, almost greedy bases in $L_1$ and $\ell_1$ behave similarly. We also notice that, since $L_1$ has no unconditional basis \cite{LinPel1968}, it has no greedy basis either.

Once realized that, for $1<p<\infty$ and $p\not=2$, there are almost greedy bases of $L_p$ with essentially different fundamental functions, the question should be determine what functions are possible fundamental functions of an almost greedy basis of $L_p$. The Radamacher type and cotype of the space shed some information in this respect. In fact, the fundamental function, say $\df$, of any unconditional for constant coefficients basis of any Banach space of type $r$ and cotype $s$ satisfies
\begin{equation}\label{eq:RadEstimates}
\pw_s \lesssim \df\lesssim \pw_r
\end{equation}
(see, e.g., \cite{AlbiacAnsorena2016}*{Proof of Lemma 2.5}). In particular, any almost greedy basis of $L_p$ satisfies \eqref{eq:RadEstimates} with $r=\min\{2,p\}$ and $s=\max\{2,p\}$. As above explained, it is known that the ends of this range are possible fundamental functions. Oddly enough, as we shall prove, these extreme functions are, up to equivalence, all fundamental functions possible for almost greedy bases of $L_p$.

\begin{theorem}\label{thm:main} Let $1<p<\infty$. If $\df$ is the fundamental function of an almost greedy basis of $L_p$, then there is $r\in\{2,p\}$ such that $\df\approx \pw_r$.
\end{theorem}

We close this introductory section by briefly describing the structure of the paper. Section~\ref{sect:main} revolves around the proof of Theorem~\ref{thm:main}. Previously, in Section~\ref{sect:preliminary}, we settle the terminology we will use, and we collect some auxiliary results that we will need.

\section{Background and terminology}\label{sect:preliminary}\noindent
Although we are mainly interested in Banach spaces, as the theory of greedy-like bases can be carry out for (not necessarily locally convex) quasi-Banach spaces (see \cite{AABW2021}), we will state the results we record in this section in this more general framework. All of them are essentially known. Nonetheless, for the reader’s ease, we will sketch some proofs. We use standard terminology on Functional Analysis and greedy-like bases as can be found in \cite{AlbiacKalton2016} and the aforementioned paper \cite{AABW2021}. For clarity, however, we record the notation that is used most heavily.

The symbol $\alpha_j\lesssim \beta_j$ for $j\in J$ means that there is a positive constant $C<\infty$ such that the families of non-negative real numbers $(\alpha_j)_{j\in J}$ and $(\beta_j)_{j\in J}$ are related by the inequality $\alpha_j\le C\beta_j$ for all $j\in J$. If $\alpha_j\lesssim \beta_j$ and $\beta_j\lesssim \alpha_j$ for $j\in J$ we say $(\alpha_j)_{j\in J}$ are $(\beta_j)_{j\in J}$ are equivalent, and we write $\alpha_j\approx \beta_j$ for $j\in J$.

Let $\XX$ be a quasi-Banach space over the real or complex scalar field $\FF$, and let $\XB=(\xx_n)_{n=1}^\infty$ be a linearly independent sequence that generates the whole space $\XX$. For a fixed sequence $\gamma=(\gamma_n)_{n=1}^\infty\in\FF^\NN$, let us consider the map
\[
S_\gamma=S_\gamma[\BB,\XX]\colon \spn( \xx_n \colon n\in\NN) \to \XX,
\quad \sum_{n=1}^\infty a_n\, \xx_n \mapsto \sum_{n=1}^\infty \gamma_n\, a_n \, \xx_n.
\]
The sequence $\XB$ is an \emph{unconditional basis} if and only if $S_\gamma$ is well-defined on $\XX$ for all $\gamma\in\ell_\infty$, and
\begin{equation}\label{eq:lu}
K_{u}=K_{u}[\BB,\XX]:=\sup_{\norm{\gamma}_\infty\le 1} \norm{S_\gamma} <\infty.
\end{equation}
If $K_u\le K<\infty$, we say that $\XB$ is $K$-unconditional.
Now, given $A\subseteq \NN$, we define the \emph{coordinate projection} onto $A$ (with respect to $\XB$) as
\[
S_A=S_A[\XB,\XX]=S_{\gamma_A}[\XB,\XX],
\]
where $\gamma_A=(\gamma_n)_{n=1}^\infty$ is the sequence defined by $\gamma_n=1$ if $n\in A$ and $\gamma_n=0$ otherwise. It is known (see, e.g., \cite{AABW2021}*{Theorem 1.10}) that $\XB$ is an unconditional basis if and only
\[
\sup \{ \norm{S_A} \colon A\subseteq\NN,\; \abs{A}<\infty\}<\infty.
\]
Set $[1,m]_\ZZ=\{n\in\ZZ \colon 1\le n\le m\}$ for $m\in\NN$. The sequence $\XB$ is a \emph{Schauder basis} if and only if it satisfies the weaker condition
\[
\sup_{m\in\NN} \norm{S_{[1,m]_\ZZ}}<\infty.
\]

A family $(f_j)_{j\in J}$ in $\XX$ is said to be \emph{semi-normalized} if $\norm{f_j}\approx 1$ for $j\in J$. For convenience, we will adopt a definition of basis that implies that only semi-normalized Schauder bases become bases. A \emph{basis} of $\XX$ will be norm-bounded sequence $\XB=(\xx_n)_{n=1}^\infty$ that generates whole space $\XX$, and for which there is a (unique) norm-bounded sequence $\XB^*:=(\xx_n^*)_{n=1}^\infty$ in the dual space $\XX^*$, called the \emph{dual basis} of $\XB$, such that $(\xx_n, \xx_n^*)_{n=1}^{\infty}$ is a biorthogonal system. A basic sequence will be a sequence in $\XX$ which is a basis of its closed linear span. Notice that, according to our terminology, any basic sequence is semi-normalized.

Let $\XB=(\xx_n)_{n=1}^\infty$ and $\YB=(\yy_n)_{n=1}^\infty$ be bases of quasi-Banach spaces $\XX$ and $\YY$, respectively. We say that $\XB$ \emph{$C$-dominates} $\YB$ if there is a linear map $T\colon \YY\to\XX$ such that $\norm{T}\le C$ and $T(\yy_n)=\xx_n$ for all $n\in\NN$. If the constant $C$ is irrelevant, we simply drop it from the notation. If $\XB$ dominates $\YB$ and, in turn, $\YB$ dominates $\XB$, we say that the bases are equivalent.

Given a basis $\XB=(\xx_n)_{n=1}^\infty$ of a quasi-Banach space $\XX$, the mapping
\begin{equation*}\label{eq:Fourier}
f\mapsto (a_n)_{n=1}^\infty:= (\xx_n^*(f))_{n=1}^{\infty}
\end{equation*}
is a bounded linear operator from $\XX$ into $c_0$, hence for each $m\in\NN$ there is a unique $A=A_m(f)\subseteq \NN$ of cardinality $\abs{A}=m$ such that whenever $n\in A$ and $j \in \NN\setminus A$, either $\abs{a_n}>\abs{a_j}$ or $\abs{a_n}=\abs{a_j}$ and $n<j$.
The \emph{$m$th greedy approximation} to $f\in\XX$ with respect to the basis $\XB$ is
\[
\GG_m(f)=\GG_m[\XB,\XX](f):=S_{A_m(f)}(f).
\]
The sequence of operators $(\GG_m)_{m=1}^{\infty}$ is called \emph{thresholding greedy algorithm} (TGA for short) on $\XX$ with respect to $\XB$. Other nonlinear operators of interest to us are the \emph{restricted truncation operators} $(\UU_{\, m})_{m=1}^{\infty}$. Let
\[
\EE=\{\varepsilon\in\FF \colon \abs{\varepsilon}=1\}.
\]
Given $A\subseteq\NN$ finite, $A\subseteq B$, $\varepsilon\in \EE^B$ we set
\[
\Ind_{\varepsilon, A}=\Ind_{\varepsilon, A}[\XB,\XX]=\sum_{n\in A} \varepsilon_n\, \xx_n.
\]
If $\varepsilon$ only takes the value $1$, we put $\Ind_{\varepsilon, A}=\Ind_A$. Given $f\in\XX$, define
\[
\varepsilon(f)= \left( \sgn (\xx_n^*(f)) \right)_{n=1}^\infty,
\]
where, as is customary, $\sgn(\cdot)$ denotes the sign function, i.e., $\sgn(0)=1$ and $\sgn(a)=a/\abs{a}$ if $a\in\FF\setminus\{0\}$. Given $m\in\NN$ we set
\[
\UU_{\, m}=\UU_{\, m}[\XB,\XX]\colon\XX\to \XX,\quad f\mapsto \left(\min_{n\in A_m(f)} \abs{\xx_n^*(f)} \right) \Ind_{\varepsilon(f), A_m(f)}.
\]

The basis $\XB$ is said to be is said to be \emph{almost greedy} if there is a constant $C$ such that
\[
\norm{f-\GG_m(f)}\le C \norm{f-S_A(f)}, \quad f\in \XX, \, m\in\NN,\, \abs{A}=m.
\]
We say that $\XB$ is \emph{greedy} if it satisfies the more demanding condition
\[
\norm{f-\GG_m(f)}\le C \norm{ f-\sum_{n\in A} a_n \, \xx_n }, \quad f\in \XX, \, m\in\NN,\, \abs{A}=m, \, a_n\in\FF.
\]

The basis $\XB$ is said to be \emph{quasi-greedy} if the TGA with respect to it is uniformly bounded. Equivalently, $\XB$ is quasi-greedy if and only if there is a constant $C\ge 1$ such that
\[
\norm{f-\GG_m(f)} \le C\norm{f}, \quad f\in\XX, \, m\in\NN.
\]

In turn, we say that $\XB$ is \emph{truncation quasi-greedy} if the restricted truncation operators are uniformly bounded, i.e., there is a constant $C$ such that
\[
\norm{\UU_m(f)} \le C\norm{f}, \quad f\in\XX, \, m\in\NN.
\]

Semi-normalized unconditional bases are a special kind of quasi-greedy bases, and although the converse is not true in general, quasi-greedy basis always retain in a certain sense a flavor of unconditionality. For example, any quasi-greedy basis is \emph{truncation quasi-greedy} (see \cite{AABW2021}*{Theorem 4.13}). In turn, if the basis $\XB=(\xx_n)_{n=1}^\infty$ is truncation quasi-greedy, then it is \emph{unconditional for constant coefficients} ( UCC for short), that is, there is a constant $C\ge 1$ such that whenever $A$, $B$ are finite subsets of $\NN$ with $A\subseteq B$ and $\varepsilon\in\EE^B$ we have $\norm{ \Ind_{\varepsilon,A}} \le C \norm{\Ind_{\varepsilon,B}}$. If the basis is UCC, then there is another constant $C\ge 1$ such that
\begin{equation}\label{eq:succ}
\norm{ \Ind_{\delta,A} } \le C \norm{ \Ind_{\varepsilon,A} }
\end{equation}
for all finite subsets $A\subseteq\NN$ and all choice of signs $\delta$ and $\varepsilon$ (see \cite{AABW2021}*{Lemma 2.2}).

A basis (or basic sequence) $\XB=(\xx_n)_{n=1}^\infty$ of a quasi-Banach space $\XX$ is said to be \emph{democratic} if there is a constant $D\ge 1$ such that
\[
\norm{\Ind_A}\le D \norm{\Ind_B}
\]
for any two finite subsets $A$, $B$ of $\NN$ with $\abs{A}\le \abs{B}$. The lack of democracy of a basis $\BB$ exhibits some sort of asymmetry. To measure how much a basis $\BB$ deviates from being democratic, we consider its \emph{upper democracy function}, also known as its \emph{fundamental function},
\[
\udf(m): = \udf[\BB, \XX](m)=\sup_{\abs{A}\le m}\norm{ \Ind_{A}} ,\quad m\in\NN,
\]
and its \emph{lower democracy function},
\[
\ldf(m):= \ldf[\BB, \XX](m)=\inf_{\abs{A}\ge m}\norm{ \Ind_{A}}, \quad m\in\NN.
\]

Suppose that $\XB$ is UCC. Then $\ldf(m)\lesssim\udf(m)$ for $m\in\NN$, hence $\XB$ is democratic if and only $\udf(m)\lesssim\ldf(m)$ for $m\in\NN$, in which case $\XB$ is \emph{super-democratic}, i.e., there is a constant $D\ge 1$ such that
\begin{equation*}
\norm{ \Ind_{\varepsilon,A}} \le C \norm{\Ind_{\varepsilon,B} }
\end{equation*}
for any two finite subsets $A$, $B$ of $\NN$ with $\abs{A}\le\abs{B}$, and any signs $\varepsilon$ and $\delta$.

Following \cite{DKKT2003}, we say that a sequence $(s_m)_{m=1}^\infty$ in $(0,\infty)$ has the upper regularity property (URP for short) if there is an integer $k>2$ such that
\[
s_{km}\le\frac{1}{2} k s_m, \quad m\in\NN.
\]
The following lemma is proved more or less explicitly in \cite{DKKT2003}.
\begin{lemma}\label{lem:URP}
Let $\sss=(s_m)_{m=1}^\infty$ be an essentially increasing sequence in $(0,\infty)$. Then, $\sss$ has the URP if and only if there is a constant $C$ such that the weight $(1/s_m)_{m=1}^\infty$ satisfies the Dini condition
\begin{equation}\label{eq:Dini}
\sum_{n=1}^m \frac{1}{s_n}\le C \frac{m}{s_m}, \quad m\in\NN.
\end{equation}
\end{lemma}

A \emph{weight} will be a sequence $\ww=(w_n)_{n=1}^\infty$ of non-negative scalars with $w_1>0$. The \emph{primitive sequence} $(s_m)_{m=1}^\infty$ of the weight $\ww$ is defined by $s_m=\sum_{n=1}^m w_n$ for all $m\in\NN$. Given $0<q\le\infty$, we will denote by $d_{1,q}(\ww)$ the Lorentz space consisting of all sequences $f\in c_0$ whose non-increasing rearrangement of $(b_n)_{n=1}^\infty$ satisfies
\[
\norm{f}_{d_{1,q}(\ww)}=\norm{( s_n\, a_n)_{n=1}^\infty}_{\ell_q(\uu)}<\infty,
\]
where $\uu=(w_n/s_n)_{n=1}^\infty$. If $\ww=n^{1/p-1}$, then $d_{1,q}(\ww)=\ell_{p,q}$ up to an equivalent quasi-norm.

Given a sequence $\ttt=(t_m)_{m=1}^\infty$ in $\NN$, the Marcinkiewicz space $m(\ttt)$ consists of all $f=(a_k)_{k=1}^\infty\in\FF^\NN$ such that
\[
\sup_{\abs{A}\le m} \frac{1}{t_m} \sum_{k\in A} \abs{a_k}<\infty.
\]
\begin{lemma}\label{lem:wLMar}
Let $\ww$ be a weight whose primitive sequence $(s_m)_{m=1}^\infty$
has the URP. Set $\ttt=(m/s_m)_{m=1}^\infty$. Then, there is a constant $C$ such that
\[
\norm{f}_{m(\ttt)}\le C
\norm{f}_{d_{1,\infty}(\ww)}, \quad f\in c_0.
\]
\end{lemma}

\begin{proof}
Use Lemma~\ref{lem:URP} to pick a constant $C$ such that \eqref{eq:Dini} holds. Fix $f=(a_k)_{k=1}^\infty\in c_0$ and $m\in\NN$. Let $(b_n)_{n=1}^\infty$ denote the non-increasing rearrangement of $f$. For any$A\subseteq\NN$ with $\abs{A}\le m$ we have
\[
\sum_{k\in A} \abs{a_k}\le \sum_{n=1}^m b_n\le \sum_{n=1}^m \frac{1}{s_n} \le C \frac{m}{s_m}.\qedhere
\]
\end{proof}

We say that a basis $\XB=(\xx_n)_{n=1}^\infty$ of a quasi-Banach space $\XX$ is \emph{bidemocratic} if there is a constant $C$ such that
\[
\udf[\XB,\XX](m) \udf[\XB^*,\XX^*](m)\le C m, \quad m\in\NN.
\]
The identity $\Ind_A[\XB^*,\XX^*](\Ind_A[\XB,\XX])=\abs{A}$ yields that if $\BB$ is bidemocratic, then both $\BB$ and $\BB^*$ are democratic, and
\begin{equation}\label{eq:udfBidem}
\udf[\XB^*,\XX^*](m)\approx\frac{m}{\udf[\XB,\XX](m) }.
\end{equation}
There is also a connection between bidemocratic bases and truncation quasi-greedy ones.

\begin{proposition}[\cite{AABW2021}*{Proposition 5.7}]\label{prop:21}
Let $\XB$ be a bidemocratic basis of a quasi-Banach space. Then $\XB$ and $\XB^*$ are truncation quasi-greedy, hence UCC and super-democratic.
\end{proposition}

Next, we bring up a partial converse of Proposition~\ref{prop:21}.
\begin{proposition}[cf.\@ \cite{DKKT2003}*{Proposition 4.4}]\label{prop:17} Let $\XB=(\xx_n)_{n=1}^\infty$ be a democratic truncation quasi greedy basis of a quasi-Banach space $\XX$. Suppose that the fundamental function of $\XB$ has the URP. Then, $\XB$ is bidemocratic.
\end{proposition}
\begin{proof} Set $\ttt=(m/s_m)_{m=1}^\infty$, where $s_m=\udf[\XB,\XX](m)$. Combining
\cite{AABW2021}*{Theorem 8.12}
with Lemma~\ref{lem:wLMar}, and taking into account the democracy of $\XB$, gives
that the unit vector system of $m(\ttt)$ $C$-dominates $\XB$ for some constant $C$.
Pick $m\in\NN$, $f\in B_\XX$ and $A\subseteq\NN$ with $\abs{A}\le m$. We have
\begin{align*}
\abs{\left( \Ind_A[\XB^*,\XX^*]\right)(f)}
\le\sum_{n\in A} \abs{\xx_n^*(f)}
\le C \frac{m}{s_m}.
\end{align*}
Taking the supremum on $f$ and $A$ yields the desired inequality.
\end{proof}

We say that a quasi-Banach space $\XX$ has Rademacher type (respectively cotype) $r$, $0<r<\infty$, if there is a constant $C$ such that for any finite family $(f_j)_{j\in A}$ in $\XX$, being
\[
A:=\Ave_{\varepsilon_j=\pm 1} \norm{\sum_{j\in A} \varepsilon_j \, f_j }\; \text{ and }
S:=\left(\sum_{j\in A} \norm{f_j}^r\right)^{1/r},
\]
we have $A\le C S$ (resp., $S\le C A$). Since the optimal type (resp., cotype) of the scalar field is $2$, if a nonzero quasi-Banach space $\XX$ has type (resp., cotype) $r$, then $r\le 2$ (resp., $r\ge 2$). Given $0<p<\infty$ and a measure space $(\Omega,\Sigma,\mu)$ such that the dimension of the vector space consisting of al integrable simple functions is infinite, the optimal type of $L_p(\mu)$ is $\min\{2,p\}$, and its optimal cotype is $\max\{2,p\}$.

Since any quasi-Banach space with a Rademacher type larger than one is locally convex \cite{Kalton1980}*{Theorem 4.1}, the following result lies within the theory of Banach spaces.
\begin{proposition}[cf.\@ \cite{DKKT2003}*{Proposition 4.1}]\label{prop:typeURP}
Let $\XB$ be a UCC basis of a Banach space $\XX$. Assume that $\XX$ has type $r>1$. Then $\df:=\udf[\XB,\XX]$ has the URP.
\end{proposition}
\begin{proof}
Let $T_r$ be the $r$-type constant of $\XX$, and let $C$ be as in \eqref{eq:succ}.
Let $k$, $m\in\NN$ and $A\subseteq \NN$ be such that $\abs{A}\le k\, m$. Pick a partition $(A_j)_{j=1}^k$ of $A$ with
$\abs{A_j}\le m$ for all $j=1$,\dots, $k$. Since $\norm{ \Ind_{A_j}}\le \df(m)$,
\[
\norm{ \Ind_A}
\le C \Ave_{\varepsilon_j=\pm 1}
\norm{ \sum_{j=1}^k \varepsilon_j \Ind_{A_j}}
\le C T_r\left( \sum_{j=1}^k \norm{ \Ind_{A_j}}^r\right)^{1/r}\le C T_r k^{1/r} \udf(m).
\]
If we choose $k$ large enough, so that $C T_r k^{1/r-1}\le 1/2$, taking the supremum on $A$ we obtain
we obtain $\df(km)\le k\df(m)/2$.
\end{proof}

\begin{corollary}\label{cor:23}
Let $\XB$ be a democratic truncation quasi-greedy basis of a Banach space $\XX$. Suppose that $\XX$ has type $r>1$. Then, $\XB$ is bidemocratic.
\end{corollary}

\begin{proof}
Just combine Proposition~\ref{prop:17} with Proposition~\ref{prop:typeURP}.
\end{proof}

\section{Greedy-type basis in $L_p$-spaces}\label{sect:main}\noindent
It is well-known that the basic-sequences structure of $L_p$ for $2<p<\infty$ is quite different from that of $L_p$ for $1<p<2$. Among others, Kadets-Pe{\l}czy{\'n}ski's milestome paper \cite{KadPel1962} brought out this asymmetry. As a matter of fact, our route toward proving Theorem~\ref{thm:main} relies on a result from this paper.
\begin{lemma}[\cite{KadPel1962}*{Corollary 5}]\label{lem:KP}
Let $2\le p<\infty$, and let $\Psi=(\bpsi_n)_{n=1}^\infty$ be a weakly-null sequence in $L_p$ with $\liminf_n \norm{\bpsi_n}_p>0$. Then, $\Psi$ has a subsequence equivalent to the unit vector system of $\ell_r$, where $r\in\{2,p\}$.
\end{lemma}

A basis $\XB$ of a quasi-Banach space $\XX$ is said to be \emph{shrinking} if its dual basis $\XB^*$ generates the whole space $\XX^*$. The space $L_p$, $1<p<\infty$, has a shrinking basis. In fact, any Schauder basis of any reflexive Banach space is shrinking \cite{James1950}.

\begin{lemma}\label{lem:Steve}
Let $(\bpsi_n)_{n=1}^\infty$ be a norm-bounded sequence in a quasi-Banach space $\XX$ with a shrinking basis $\XB=(\xx_j)_{j=1}^\infty$. Then, there are sequences $(n_j)_{j=1}^\infty$ and $(m_j)_{j=1}^\infty$ in $\NN$ such that $n_j<m_j<n_{j+1}$ for all $j\in\NN$, and $(\bpsi_{n_j}-\bpsi_{m_j})_{j=1}^\infty$ is weakly null.
\end{lemma}

\begin{proof}
For each $j\in\NN$, the sequence $\alpha_j:=(\xx_j^*(\bpsi_n))_{n=1}^\infty$ is bounded. By the diagonal Cantor technique, passing to a subsequence we can suppose that $\alpha_j$ converges for every $j\in\NN$. Set $\bphi_n=\bpsi_{2n}-\bpsi_{2n+1}$ for all $n\in\NN$. We have $\lim_n \xx_j^*(\bphi_n)=0$ for all $j\in\NN$. A standard approximation argument yields $\lim_n f^*(\bphi_n)=0$ for every $f^*\in\XX^*$.
\end{proof}

\begin{lemma}\label{lem:37}
Let $2\le p<\infty$, and let $\Psi=(\bpsi_n)_{n=1}^\infty$ be a uniformly separated bounded sequence in $L_p$. Then, there are $r\in\{2,p\}$ and sequences $(n_j)_{j=1}^\infty$ and $(m_j)_{j=1}^\infty$ in $\NN$ with $n_j<m_j<n_{j+1}$ for all $j\in\NN$ such that $(\bpsi_{n_j}-\bpsi_{m_j})_{j=1}^\infty$ is equivalent to the unit vector system of $\ell_r$.
\end{lemma}

\begin{proof}
Just combine Lemma~\ref{lem:KP} with Lemma~\ref{lem:Steve}.
\end{proof}

For further reference, we record an elementary result about democracy functions. We omit its straightforward proof.
\begin{lemma}\label{lem:SBlp}
If a basis $\XB$ of a quasi-Banach space $\XX$ has a subbasis equivalent to the unit vector system of $\ell_r$, $0<r\le\infty$, then
\[
\ldf[\XB,\XX]\lesssim \pw_r \lesssim \udf[\XB,\XX].
\]
\end{lemma}

\begin{proposition}\label{prop:2<p}
Let $2\le p<\infty$, and let $\Psi=(\bpsi_n)_{n=1}^\infty$ be a UCC basic sequence in $L_p$. Then, for either $r=2$ or $r=p$,
\begin{equation*}
\ldf[\Psi,L_p]\lesssim \pw_r \lesssim \udf[\Psi,L_p].
\end{equation*}
\end{proposition}

\begin{proof}
By Lemma~\ref{lem:37} and Lemma~\ref{lem:SBlp}, there are $r\in\{2,p\}$ and $(n_j)_{j=1}^\infty$ and $(m_j)_{j=1}^\infty$ in $\NN$ with $n_j<m_j<n_{j+1}$ for all $j\in\NN$ such that the basic sequence $\Phi=(\bpsi_{n_j}-\bpsi_{m_j})_{j=1}^\infty$ satisfies
\[
\ldf[\Phi,L_p] \lesssim \pw_r \lesssim \udf[\Phi,L_p].
\]
Since $\Psi$ is UCC, $\ldf[\Psi,L_p]\lesssim \ldf[\Phi,L_p]$. In turn, since $\udf[\Psi,L_p]$ is doubling (see, e.g., \cite{AABW2021}*{\S8}), $\udf[\Phi,L_p]\lesssim \udf[\Psi,L_p]$.
\end{proof}

We can also prove Proposition~\ref{prop:2<p} using ideas from \cite{AACV2019}. Let us detail this alternative approach.
\begin{lemma}\label{lem:1}
Let $(\Omega,\Sigma,\mu)$ be a finite measure space and let $1\le p<\infty$. Suppose that $\Psi=(\bpsi_n)_{n=1}^\infty$ is an UCC basic sequence in $L_p(\mu)$ with
$\norm{\bpsi_n}_2\approx 1$ for $n\in\NN$. Then, $\Psi$ is democratic, and its fundamental function grows as $\pw_2$.
\end{lemma}

\begin{proof}
The proof of \cite{AACV2019}*{Proposition 2.5} remains valid in this slightly more general framework. In fact, all we need to prove \cite{AACV2019}*{Lemma 2.1} is that the involved basis is UCC.
\end{proof}

We will use a gliding-hump-type lemma for $L_p$-spaces.
\begin{lemma}\label{lem:2}
Let $0<q<p<\infty$, let $(\Omega,\Sigma,\mu)$ be a finite measure space, and let $(\epsilon_k)_{k=1}^\infty$ be a sequence of positive scalars. For each $n\in\NN$, let $\bpsi_n\colon\Omega\to\FF$ be a measurable function. Suppose that $\sup_n \norm{\bpsi_n}_p<\infty$ and $\liminf_n \norm{\bpsi_n}_q=0$. Then, there is an increasing sequence $(n_k)_{k=1}^\infty$ of positive integers and a sequence $(A_k)_{k=1}^\infty$ of pairwise disjoint sets in $\Sigma$ such that
\[
\norm{\bpsi_{n_k} -\bpsi_{n_k}\chi_{A_k}}_p \le \epsilon_k, \quad k\in\NN.
\]
\end{lemma}
\begin{proof}
Replacing $\bpsi_n$ with $\abs{\bpsi_n}^p$, and passing to suitable subsequence, we can assume that $p=1$ and $\lim_n \norm{\bpsi_n}_q=0$. Then, $(\bpsi_n)_{n=1}^\infty$ converges to zero in measure. Applying \cite{AlbiacKalton2016}*{Theorem 5.2.1}, and passing to a further subsequence, yields the desired result.
\end{proof}

In light of Lemma~\ref{lem:SBlp}, we can derive Proposition~\ref{prop:2<p} from Lemma~\ref{lem:2<pBis} below.

\begin{lemma}\label{lem:2<pBis}
Let $2\le p<\infty$, and let $\Psi=(\bpsi_n)_{n=1}^\infty$ be a UCC basic sequence in $L_p$. Then, either
\begin{equation}\label{eq:Deml2}
\ldf[\Psi,L_p]\approx \pw_2\approx \udf[\Psi,L_p]
\end{equation}
or $\Psi$ has a subbasis equivalent to the unit vector system of $\ell_p$.
Moreover, in the case when $p>2$, \eqref{eq:Deml2} holds if and only if $\inf_n \norm{\bpsi_n}_2>0$.
\end{lemma}
\begin{proof}
We shall consider two opposite situations.

\noindent$\bullet$ Suppose that $\inf_n \norm{\bpsi_n}_2>0$. Then, \eqref{eq:Deml2} holds by Lemma~\ref{lem:1}.

\noindent$\bullet$ Suppose that $\inf_n \norm{\bpsi_n}_2=0$. Pick a sequence $(\epsilon_k)_{k=1}^\infty$ such that
\[
0<\epsilon_k<c:=\inf_n \norm{\bpsi_n}_p,\quad k\in\NN,
\]
and
\[
\sum_{k=1}^\infty \frac{\epsilon_k}{(c^p-\epsilon_k^p)^{1/p}}<1.
\]
By Lemma~\ref{lem:2} there is pairwise disjoint sequence $(A_k)_{k=1}^\infty$ consisting of measurable sets and an increasing sequence $(n_k)_{k=1}^\infty$ in $\NN$ such that
\[
\norm{ \bpsi_{n_k} -\bphi_k }_p\le\epsilon_k, \quad k\in\NN,
\]
where $\bphi_k=\bpsi_{n_k} \chi_{A_k} $. We have
\[
(c^p-\epsilon_k^p)^{1/p} \le \norm{\bphi_k }_p \le \norm{\bpsi_{n_k}}_p, \quad k\in\NN,
\]
whence we infer that $\Phi:=(\bphi_k)_{k=1}^\infty$ semi-normalized. Therefore, $\Phi$ is a $1$-unconditional basic sequence equivalent to the unit vector system of $\ell_p$. In turn, by the principle of small perturbations (see, e.g., \cite{AlbiacKalton2016}*{Theorem 1.3.9}), $(\bpsi_{n_k})_{k=1}^\infty$ is equivalent to $\Phi$.
\end{proof}

Let $1<p<2$. Since $\ell_r$ is a subspace $L_p$ for every $p\le r\le 2$ \cite{Rosenthal1972}, $L_p$ has, for each $r\in[p,2]$, a greedy basic sequence whose fundamental function grows as $\pw_r$. The situation in the case where $p>2$ is quite different.

\begin{corollary}\label{cor:38}
Let $2\le p<\infty$, and let $\df$ be the fundamental function of a super-democratic basic sequence in $L_p$. Then, either $\df\approx\pw_2$ or $\df\approx\pw_p$.
\end{corollary}

\begin{proof}
It is a straightforward consequence of Proposition~\ref{prop:2<p}.
\end{proof}

Our approach also leads to settling the structure of subsymmetric basic sequences of $L_p$, $p>2$. Although this result is probably well-known to specialists, we make a detour on our route to record it. Recall that a sequence in a Banach space is said to be a subsymmetric basis if it is an unconditional basis and it is equivalent to all its subsequences. All we need to know about this important class of bases is the following.

\begin{lemma}[see \cite{AADK2021}*{Lemma 2.2}]\label{lem:SubSym}
Let $\XB=(\xx_n)_{n=1}^\infty$ be a subsymmetric basis of a quasi-Banach space.
Let $(n_j)_{j=1}^\infty$ and $(m_j)_{j=1}^\infty$ in $\NN$ be sequences in $\NN$ with $n_j<m_j<n_{j+1}$ for all $j\in\NN$.
Then, $(\xx_{n_j}-\xx_{m_j})_{n=1}^\infty$ is equivalent to $\XB$.
\end{lemma}

\begin{theorem}
Let $\Psi$ be a subsymmetric basic sequence in $L_p$, $2\le p<\infty$. Then, $\Psi$ is equivalent the unit vector system of either $\ell_2$ or $\ell_p$.
\end{theorem}

\begin{proof}
Just combine Lemma~\ref{lem:SubSym} with Lemma~\ref{lem:37}.
\end{proof}

A Banach space $\XX$ is said to be an $\SL_p$-space, $1\le p\le\infty$, if for every finite-dimensional subspace $\VV\subseteq\XX$ there is a further finite-dimen\-sional subspace $\VV\subseteq\WW\subseteq \XX$ whose Banach-Mazur distance to $\ell_p^{\dim(\WW)}$ is uniformly bounded. It is known that, in the case when $1<p<\infty$, $\XX$ is a separable $\SL_p$-space if and only if it is isomorphic to a non-Hilbertian complemented subspace of $L_p$ (see \cite{LinRos1969}).

Given $1\le p\le\infty$, we denote by $p'$ its conjugate exponent defined by $1/p'=1-1/p$. Given $A\subseteq[1,\infty]$, we set $A'=\{p'\colon p\in A\}$.
\begin{corollary}\label{cor:44}
Let $1<p<\infty$, and let $\XB$ be a bidemocratic basis of an $\SL_p$-space $\XX$. Then, $\ldf[\XB,\XX]\approx \pw_r \approx \udf[\XB,\XX]$, where either $r=p$ or $r=2$.
\end{corollary}

\begin{proof}
If $2\le p<\infty$ the result follows from combining Proposition~\ref{prop:21} with Corollary~\ref{cor:38}. Assume that $1<p<2$. Then, by Proposition~\ref{prop:21}, $\XB^*$ is a super-democratic basis of a Banach space isomorphic to a subspace of $L_p^*$. Since $L_p^*$ is isometrically isomorphic to $L_{p'}$, there is $s\in\{2,p'\}$ such that $\udf[\XB^*,\XX^*]\approx \pw_s$. Consequently, by Equation~\eqref{eq:udfBidem}, $\udf[\XB, \XX]\approx \pw_r$ for some $r\in \{2,p'\}'=\{2,p\}$.
\end{proof}

Now, we obtain Theorem~\ref{thm:main} as a consequence of the following more general result.
\begin{theorem}
Let $1<p<\infty$, and let $\df$ be the fundamental function of a super-democratic basis of an $\SL_p$-space. Then, $\df\approx \pw_r$, where either $r=p$ or $r=2$.
\end{theorem}

\begin{proof}
Just combine Corollary~\ref{cor:23} with Corollary~\ref{cor:44}.
\end{proof}

We close the paper by exhibiting an application of Theorem~\ref{thm:main}. Let $\YY$ be a \emph{symmetric space}, i.e., a quasi-Banach space $\YY\subseteq\FF^\NN$ for which the unit vector system is a symmetric basis. We say that a quasi-Banach space $\XX$ with a basis $\XB$ embeds in $\YY$ via $\XB$, and we write $\XX \stackrel{\XB}\hookrightarrow \YY$, if $\XB$ dominates unit vector system of $\YY$. In the reverse direction, we say that $\YY$ embeds in $\XX$ via $\XB$, and we write $\YY \stackrel{\XB}\hookrightarrow \XX$, if the unit vector system of $\YY$ dominates $\XB$. Let $\YY_1$ and $\YY_2$ be symmetric spaces whose unit vector systems have equivalent fundamental functions. Squeezing the Banach space $\XX$ as
\[
\YY_1\stackrel{\XB}\hookrightarrow \XX\stackrel{\XB}\hookrightarrow \YY_2
\]
is a condition that guarantees in a certain sense the optimality of the compression algorithms with respect to the basis $\XB$. Besides, such embeddings serve in some situations as a tool to derive other properties of the basis $\XB$. We refer the reader to \cite{AABW2021}*{\S9} for details.

\begin{corollary}
Let $\Psi$ be an almost greedy basis of $L_p$, $1<p<\infty$. Then, there are $r\in\{2,p\}$ and $1<q<s<\infty$ such that
\[
\ell_{r,q} \stackrel{\Psi}\hookrightarrow L_p \stackrel{\Psi}\hookrightarrow \ell_{r,s}.
\]
\end{corollary}

\begin{proof}
Combine Theorem~\ref{thm:main} with \cite{ABW2021}*{Theorem 1.1}.
\end{proof}


\begin{bibdiv}
\begin{biblist}

\bib{AlbiacAnsorena2016}{article}{
      author={Albiac, Fernando},
      author={Ansorena, Jos\'{e}~L.},
       title={Lorentz spaces and embeddings induced by almost greedy bases in
  {B}anach spaces},
        date={2016},
        ISSN={0176-4276},
     journal={Constr. Approx.},
      volume={43},
      number={2},
       pages={197\ndash 215},
         url={https://doi-org/10.1007/s00365-015-9293-3},
      review={\MR{3472645}},
}

\bib{AABW2021}{article}{
      author={Albiac, Fernando},
      author={Ansorena, Jos\'{e}~L.},
      author={Bern\'{a}, Pablo~M.},
      author={Wojtaszczyk, Przemys{\l}aw},
       title={Greedy approximation for biorthogonal systems in quasi-{B}anach
  spaces},
        date={2021},
     journal={Dissertationes Math. (Rozprawy Mat.)},
      volume={560},
       pages={1\ndash 88},
}

\bib{AACV2019}{article}{
      author={Albiac, Fernando},
      author={Ansorena, Jos\'{e}~L.},
      author={Ciaurri, \'{O}scar},
      author={Varona, Juan~L.},
       title={Unconditional and quasi-greedy bases in {$L_p$} with applications
  to {J}acobi polynomials {F}ourier series},
        date={2019},
        ISSN={0213-2230},
     journal={Rev. Mat. Iberoam.},
      volume={35},
      number={2},
       pages={561\ndash 574},
         url={https://doi.org/10.4171/rmi/1062},
      review={\MR{3945734}},
}

\bib{AADK2016}{article}{
      author={Albiac, Fernando},
      author={Ansorena, Jos\'{e}~L.},
      author={Dilworth, Stephen~J.},
      author={Kutzarova, Denka},
       title={Banach spaces with a unique greedy basis},
        date={2016},
        ISSN={0021-9045},
     journal={J. Approx. Theory},
      volume={210},
       pages={80\ndash 102},
         url={http://dx.doi.org/10.1016/j.jat.2016.06.005},
      review={\MR{3532713}},
}

\bib{AADK2019b}{article}{
      author={Albiac, Fernando},
      author={Ansorena, Jos\'{e}~L.},
      author={Dilworth, Stephen~J.},
      author={Kutzarova, Denka},
       title={Building highly conditional almost greedy and quasi-greedy bases
  in {B}anach spaces},
        date={2019},
        ISSN={0022-1236},
     journal={J. Funct. Anal.},
      volume={276},
      number={6},
       pages={1893\ndash 1924},
         url={https://doi-org/10.1016/j.jfa.2018.08.015},
      review={\MR{3912795}},
}

\bib{AADK2021}{article}{
      author={Albiac, Fernando},
      author={Ansorena, Jos\'{e}~L.},
      author={Dilworth, Stephen~J.},
      author={Kutzarova, Denka},
       title={A dichotomy for subsymmetric basic sequences with applications to
  {G}arling spaces},
        date={2021},
        ISSN={0002-9947},
     journal={Trans. Amer. Math. Soc.},
      volume={374},
      number={3},
       pages={2079\ndash 2106},
         url={https://doi.org/10.1090/tran/8278},
      review={\MR{4216733}},
}

\bib{AAW2021b}{article}{
      author={Albiac, Fernando},
      author={Ansorena, Jos\'{e}~L.},
      author={Wojtaszczyk, Przemys{\l}aw},
       title={On certain subspaces of {$\ell_p$} for {$0<p\leq1$} and their
  applications to conditional quasi-greedy bases in {$p$}-{B}anach spaces},
        date={2021},
        ISSN={0025-5831},
     journal={Math. Ann.},
      volume={379},
      number={1-2},
       pages={465\ndash 502},
         url={https://doi-org/10.1007/s00208-020-02069-3},
      review={\MR{4211094}},
}

\bib{AAW2021}{article}{
      author={Albiac, Fernando},
      author={Ansorena, Jos\'{e}~L.},
      author={Wojtaszczyk, Przemys{\l}aw},
       title={Quasi-greedy bases in {$\ell_ p$} {$(0<p<1)$} are democratic},
        date={2021},
        ISSN={0022-1236},
     journal={J. Funct. Anal.},
      volume={280},
      number={7},
       pages={108871, 21},
         url={https://doi-org/10.1016/j.jfa.2020.108871},
      review={\MR{4211033}},
}

\bib{AlbiacKalton2016}{book}{
      author={Albiac, Fernando},
      author={Kalton, Nigel~J.},
       title={Topics in {B}anach space theory},
     edition={Second Edition},
      series={Graduate Texts in Mathematics},
   publisher={Springer, [Cham]},
        date={2016},
      volume={233},
        ISBN={978-3-319-31555-3; 978-3-319-31557-7},
         url={https://doi.org/10.1007/978-3-319-31557-7},
        note={With a foreword by Gilles Godefroy},
      review={\MR{3526021}},
}

\bib{ABW2021}{article}{
      author={Ansorena, Jos\'{e}~L.},
      author={Bello, Glenier},
      author={Wojtaszczyk, Przemys{\l}aw},
       title={Lorentz spaces and embeddings induced by almost greedy bases in
  superreflexive {B}anach spaces},
        date={2021},
     journal={arXiv e-prints},
      eprint={2105.09203},
        note={Accepted for publication in Israel Journal of Mathematics},
}

\bib{BCLT1985}{article}{
      author={Bourgain, Jean},
      author={Casazza, Peter~G.},
      author={Lindenstrauss, Joram},
      author={Tzafriri, Lior},
       title={Banach spaces with a unique unconditional basis, up to
  permutation},
        date={1985},
        ISSN={0065-9266},
     journal={Mem. Amer. Math. Soc.},
      volume={54},
      number={322},
       pages={iv+111},
         url={https://doi-org/10.1090/memo/0322},
      review={\MR{782647}},
}

\bib{CasKal1998}{article}{
      author={Casazza, Peter~G.},
      author={Kalton, Nigel~J.},
       title={Uniqueness of unconditional bases in {B}anach spaces},
        date={1998},
        ISSN={0021-2172},
     journal={Israel J. Math.},
      volume={103},
       pages={141\ndash 175},
         url={https://doi-org/10.1007/BF02762272},
      review={\MR{1613564}},
}

\bib{DHK2006}{article}{
      author={Dilworth, Stephen~J.},
      author={Hoffmann, Mark},
      author={Kutzarova, Denka},
       title={Non-equivalent greedy and almost greedy bases in {$l_p$}},
        date={2006},
        ISSN={0972-6802},
     journal={J. Funct. Spaces Appl.},
      volume={4},
      number={1},
       pages={25\ndash 42},
         url={https://doi-org/10.1155/2006/368648},
      review={\MR{2194634}},
}

\bib{DKK2003}{article}{
      author={Dilworth, Stephen~J.},
      author={Kalton, Nigel~J.},
      author={Kutzarova, Denka},
       title={On the existence of almost greedy bases in {B}anach spaces},
        date={2003},
        ISSN={0039-3223},
     journal={Studia Math.},
      volume={159},
      number={1},
       pages={67\ndash 101},
         url={https://doi.org/10.4064/sm159-1-4},
        note={Dedicated to Professor Aleksander Pe{\l}czy\'nski on the occasion
  of his 70th birthday},
      review={\MR{2030904}},
}

\bib{DKKT2003}{article}{
      author={Dilworth, Stephen~J.},
      author={Kalton, Nigel~J.},
      author={Kutzarova, Denka},
      author={Temlyakov, Vladimir~N.},
       title={The thresholding greedy algorithm, greedy bases, and duality},
        date={2003},
        ISSN={0176-4276},
     journal={Constr. Approx.},
      volume={19},
      number={4},
       pages={575\ndash 597},
         url={https://doi-org/10.1007/s00365-002-0525-y},
      review={\MR{1998906}},
}

\bib{DSBT2012}{article}{
      author={Dilworth, Stephen~J.},
      author={Soto-Bajo, Mois\'es},
      author={Temlyakov, Vladimir~N.},
       title={Quasi-greedy bases and {L}ebesgue-type inequalities},
        date={2012},
        ISSN={0039-3223},
     journal={Studia Math.},
      volume={211},
      number={1},
       pages={41\ndash 69},
         url={https://doi-org/10.4064/sm211-1-3},
      review={\MR{2990558}},
}

\bib{James1950}{article}{
      author={James, Robert~C.},
       title={Bases and reflexivity of {B}anach spaces},
        date={1950},
        ISSN={0003-486X},
     journal={Ann. of Math. (2)},
      volume={52},
       pages={518\ndash 527},
         url={https://doi-org/10.2307/1969430},
      review={\MR{39915}},
}

\bib{KadPel1962}{article}{
      author={Kadets, Mikhail~I.},
      author={Pe{\l}czy{\'n}ski, Aleksander},
       title={Bases, lacunary sequences and complemented subspaces in the
  spaces {$L_{p}$}},
        date={1961/1962},
        ISSN={0039-3223},
     journal={Studia Math.},
      volume={21},
       pages={161\ndash 176},
      review={\MR{0152879}},
}

\bib{Kalton1977b}{article}{
      author={Kalton, Nigel~J.},
       title={Orlicz sequence spaces without local convexity},
        date={1977},
        ISSN={0305-0041},
     journal={Math. Proc. Cambridge Philos. Soc.},
      volume={81},
      number={2},
       pages={253\ndash 277},
         url={https://doi-org/10.1017/S0305004100053342},
      review={\MR{433194}},
}

\bib{Kalton1980}{article}{
      author={Kalton, Nigel~J.},
       title={Convexity, type and the three space problem},
        date={1980/81},
        ISSN={0039-3223},
     journal={Studia Math.},
      volume={69},
      number={3},
       pages={247\ndash 287},
         url={https://doi-org/10.4064/sm-69-3-247-287},
      review={\MR{647141}},
}

\bib{KLW1990}{article}{
      author={Kalton, Nigel~J.},
      author={Ler\'{a}noz, Camino},
      author={Wojtaszczyk, Przemys{\l}aw},
       title={Uniqueness of unconditional bases in quasi-{B}anach spaces with
  applications to {H}ardy spaces},
        date={1990},
        ISSN={0021-2172},
     journal={Israel J. Math.},
      volume={72},
      number={3},
       pages={299\ndash 311 (1991)},
         url={https://doi.org/10.1007/BF02773786},
      review={\MR{1120223}},
}

\bib{KoTe1999}{article}{
      author={Konyagin, Sergei~V.},
      author={Temlyakov, Vladimir~N.},
       title={A remark on greedy approximation in {B}anach spaces},
        date={1999},
        ISSN={1310-6236},
     journal={East J. Approx.},
      volume={5},
      number={3},
       pages={365\ndash 379},
      review={\MR{1716087}},
}

\bib{KotheToeplitz1934}{article}{
      author={K{\"o}the, Gottfried},
      author={Toeplitz, Otto},
       title={Lineare {R}{\"a}ume mit unendlich vielen {K}oordinaten und
  {R}inge unendlicher {M}atrizen},
        date={1934},
        ISSN={0075-4102},
     journal={J. Reine Angew. Math.},
      volume={171},
       pages={193\ndash 226},
         url={https://doi-org/10.1515/crll.1934.171.193},
      review={\MR{1581429}},
}

\bib{LinPel1968}{article}{
      author={Lindenstrauss, Joram},
      author={Pe{\l}czy\'{n}ski, Aleksander},
       title={Absolutely summing operators in {$L_{p}$}-spaces and their
  applications},
        date={1968},
        ISSN={0039-3223},
     journal={Studia Math.},
      volume={29},
       pages={275\ndash 326},
         url={https://doi-org/10.4064/sm-29-3-275-326},
      review={\MR{0231188}},
}

\bib{LinRos1969}{article}{
      author={Lindenstrauss, Joram},
      author={Rosenthal, Haskell~P.},
       title={The {$\mathcal{L}_{p}$} spaces},
        date={1969},
        ISSN={0021-2172},
     journal={Israel J. Math.},
      volume={7},
       pages={325\ndash 349},
         url={https://doi-org/10.1007/BF02788865},
      review={\MR{0270119}},
}

\bib{LinZip1969}{article}{
      author={Lindenstrauss, Joram},
      author={Zippin, Mordecay},
       title={Banach spaces with a unique unconditional basis},
        date={1969},
     journal={J. Functional Analysis},
      volume={3},
       pages={115\ndash 125},
         url={https://doi-org/10.1016/0022-1236(69)90054-8},
      review={\MR{0236668}},
}

\bib{Nielsen2007}{article}{
      author={Nielsen, Morten},
       title={An example of an almost greedy uniformly bounded orthonormal
  basis for {$L_p(0,1)$}},
        date={2007},
        ISSN={0021-9045},
     journal={J. Approx. Theory},
      volume={149},
      number={2},
       pages={188\ndash 192},
         url={https://doi-org/10.1016/j.jat.2007.04.011},
      review={\MR{2374604}},
}

\bib{Rosenthal1972}{article}{
      author={Rosenthal, Haskell~P.},
       title={On subspaces of {$L^{p}$}},
        date={1973},
        ISSN={0003-486X},
     journal={Ann. of Math. (2)},
      volume={97},
       pages={344\ndash 373},
         url={https://doi-org/10.2307/1970850},
      review={\MR{312222}},
}

\bib{Smela2007}{article}{
      author={Smela, Krzysztof},
       title={Subsequences of the {H}aar basis consisting of full levels in
  {$H_p$} for {$0<p<\infty$}},
        date={2007},
        ISSN={0002-9939},
     journal={Proc. Amer. Math. Soc.},
      volume={135},
      number={6},
       pages={1709\ndash 1716},
         url={https://doi-org/10.1090/S0002-9939-06-08616-3},
      review={\MR{2286080}},
}

\bib{Woj1997}{article}{
      author={Wojtaszczyk, Przemys{\l}aw},
       title={Uniqueness of unconditional bases in quasi-{B}anach spaces with
  applications to {H}ardy spaces. {II}},
        date={1997},
        ISSN={0021-2172},
     journal={Israel J. Math.},
      volume={97},
       pages={253\ndash 280},
         url={https://doi-org/10.1007/BF02774040},
      review={\MR{1441252}},
}

\bib{Woj2000}{article}{
      author={Wojtaszczyk, Przemys{\l}aw},
       title={Greedy algorithm for general biorthogonal systems},
        date={2000},
        ISSN={0021-9045},
     journal={J. Approx. Theory},
      volume={107},
      number={2},
       pages={293\ndash 314},
         url={https://doi-org/10.1006/jath.2000.3512},
      review={\MR{1806955}},
}

\end{biblist}
\end{bibdiv}



\end{document}